\documentclass[12pt,twoside,a4paper]{amsart}
\usepackage{amssymb}
\usepackage{amsmath,amscd}
\usepackage{epic}

\newtheorem{thm}{Theorem}[section]
\newtheorem{lma}[thm]{Lemma}
\newtheorem{cor}[thm]{Corollary}
\newtheorem{prop}[thm]{Proposition}

\theoremstyle{definition}

\newtheorem{df}{Definition}

\theoremstyle{remark}

\newtheorem{preremark}{Remark}
\newtheorem{preex}{Example}

\numberwithin{equation}{section}

\begin{document}

\title{Tight holomorphic maps, a classification}

\date{\today}

\author{Oskar Hamlet}

\address{Department of Mathematics\\Chalmers University of Technology and the University of Gothenburg\\412 96 G\"OTEBORG\\SWEDEN}

\email{hamlet@chalmers.se}



\begin{abstract}
We classify all tight holomorphic maps between Hermitian symmetric spaces.
\end{abstract}

\maketitle

\section{Introduction}
Let $\Gamma=\pi_1(\Sigma_g)$ be the fundamental group of an oriented compact surface $\Sigma_g$ of genus $g$ and $G=Isom(\mathcal{X})$ the isometry group of a Hermitian symmetric space of noncompact type $\mathcal{X}$. A subject of a lot of recent interest is the study of representations $\rho\colon \Gamma \rightarrow G$. In \cite{A9} Burger, Wienhard and Iozzi showed that maximal representations, that is representations with maximal Toledo invariant, have some striking geometrical properties. If $H$ is another Hermittian Lie group and $\eta\colon G \rightarrow H$ a positive homomorphism, then $\eta \circ \rho$ is maximal if and only if $\rho$ is maximal and $\eta$ is tight. Tight representations correspond to tight maps between the corresponding Hermitian symmetric spaces. With the exception of the irreducible representations $\rho_n\colon SU(1,1)\rightarrow Sp(2n)$ all known tight representations between Hermitian Lie groups correspond to tight holomorphic maps. In this paper we classify all tight holomorphic maps between Hermitian symmetric spaces.

\section{Description of the problem}\label{descr}
Let $\mathcal{X}_1=G_1/K_1, \mathcal{X}_2=G_2/K_2$ be Hermitian symmetric spaces of the non-compact type. Equip $\mathcal{X}_i$ with the unique $G_i$-invariant metric $g_i$ such that the minimal holomorphic sectional curvature is $-1$ on each irreducible factor of $\mathcal{X}_i$. Further let  $\omega_i$ be the associated K\"{a}hler forms defined by $\omega_i(X,Y):=g_i(J_iX,Y)$, where $J_i$ denotes the complex structure of $\mathcal{X}_i$ .
Let $f \colon \mathcal{X}_1\rightarrow \mathcal{X}_2$ be a totally geodesic map, i.e the image of geodesics in $\mathcal{X}_1$ are geodesics in $\mathcal{X}_2$. We have  
\begin{equation}\label{tightdef}
sup_{\Delta\subset\mathcal{X}_1}\int_\Delta{f^*\omega_2}\leq sup_{\Delta\subset\mathcal{X}_2}\int_\Delta{\omega_2}
\end{equation}
where the supremum is taken over all triangles with geodesic sides. The inequality is obvious since $f(\Delta)$  is a geodesic triangle in $\mathcal{X}_2$ for $\Delta$ a geodesic triangle in $\mathcal{X}_1$. If we have equality in (\ref{tightdef}) we say that the map is \emph{tight} \cite{A8}.\\
Our objective is to classify tight holomorphic maps up to equivalence. We say that two maps are equivalent if they are equal up to an action of an element of $G_2$. 

It will be useful to look at the problem from different perspectives. Let $\mathfrak{g}_i$ be the Lie algebra of $G_i$ and $\mathfrak{g}_i=\mathfrak{k}_i+\mathfrak{p}_i$ a Cartan decompostion. Let $0$ be a fixed basepoint of $\mathcal{X}_i$. Identifying $T_{0}\mathcal{X}_i$ with $\mathfrak{p}_i$, totally geodesic maps correspond to Lie algebra homomorphisms respecting the Cartan decomposition. As is well known there exists an element $Z_i$ in the center of $\mathfrak{k}_i$ such that the complex structure on $\mathfrak{p}_i$ is given by $ad(Z_i)$. A map $f$ is holomorphic if and only if the corresponding Lie algebra homomorphism $\rho$ satisfies
\begin{equation*}\tag{H1}
\rho \circ ad(Z_1)=ad(Z_2)\circ \rho. 
\end{equation*}
A stronger condition, that  
\begin{equation*}\tag{H2}
\rho(Z_1)=Z_2
\end{equation*}
will also play an important role. See \cite{A6} and \cite{A7} for a clarification of the roles of (H1) and (H2) in holomorphic representations. Our goal is reformulated as classifying all Lie algebra homomorphisms satisfying (H1) and corresponding to tight maps. We will later on give conditions in Lie algebraic terms for a homomorphism to correspond to a tight map.\\
Yet another point of view is that of the corresponding homomorphisms of the isometry groups and the maps induced by these on continuous bounded cohomology. In a familiar way $\omega_i$ defines a cohomology class $\kappa_{G_i}^b\in H^2_{cb}(G_i)$, defined by $\kappa_{G_i}^b(g_0,g_1,g_2):=\int_{\Delta(g_0 \cdot 0,g_1\cdot 0,g_2\cdot 0)}{\omega_i}$ where $\Delta(g_0\cdot 0,g_1\cdot 0,g_2\cdot 0)$ is the geodesic triangle with vertices $g_0\cdot 0,g_1\cdot 0,g_2\cdot 0$. Continuous bounded cohomology is equipped with a canonical seminorm, namely the supremum norm. For a continuous representation $\rho:G_1\rightarrow G_2$, $\rho^*$ is norm decreasing, i.e $||\rho^*\kappa_{G_2}^b||\leq ||\kappa_{G_2}^b||$. If there is equality we say that $ \rho$ is a tight homomorphism. Each $\rho$ corresponds to a totally geodesic map $f$ and we have, as the name suggests, that $\rho$ is tight if and only if $f$ is tight \cite{A8}.
We can thus view the problem on three different levels. The symmetric space, the isometry group and the corresponding Lie algebra.  We will switch frequently between these perspectives as they all have their advantages. While doing so we will abuse the concepts and notation a bit. We will talk about holomorphic representations when what we mean is that the corresponding totally geodesic map is holomorphic. We will also use the same letter to denote a homomorphism and the corresponding totally geodesic map.

We divide the classification into four parts. In the section 3 we reduce the problem to the classification of tight (H2)-representations and tight regular subalgebras. In the section 4 we give the tools needed to determine if a representation or subalgebra is tight. In section 5 we classify the tight regular subalgebras and in section 6 we classify the tight (H2)-representations.

\section{Reduction of the problem} \label{sectrescur}
The reduction consists of three main lemmas. We start by stating these and explaining their role in the reduction. We then recall some facts about bounded cohomology and prove Lemmas \ref{lma1} and \ref{lma2}. We then give the definition of regular subalgebras. We will not attempt to prove Lemma \ref{lma3} here but refer the reader to \cite{A6} instead.
\begin{lma}\label{lma1}
Let $\rho\colon \mathcal{Y}\rightarrow \mathcal{X}_1\times ... \times \mathcal{X}_n$ be a holomorphic embedding. Denote by $\pi_i$ the projection map onto the $\mathcal{X}_i$. Then $\rho$ is tight if and only if all the $\pi_i\circ\rho$ are tight. 
\end{lma}
Lemma \ref{lma1} allows us to reduce our classification to the case when the target space is irreducible.
\begin{lma}\label{lma2}
Let $\rho\colon \mathcal{X}\rightarrow \mathcal{Y}$ and $\eta\colon \mathcal{Y}\rightarrow \mathcal{Z}$ be totally geodesic holomorphic maps and further assume that $\eta$ is injective. Then $\eta\circ\rho$ is tight if and only if $\rho$ and $\eta$ are tight.
\end{lma}
Lemma \ref{lma2} quickly reduces our classification to injective maps. Let $\rho\colon\mathcal{X}\rightarrow\mathcal{X}^{'}$ be a non-injective holomorphic totally geodesic map. Let $\rho\colon \mathfrak{g}\rightarrow \mathfrak{g}^{'}$ be the corresponding lie algebra homomorphism. The kernel of $\rho$ is then an ideal $\mathfrak{g}_1$ of $\mathfrak{g}$ corresponding to a subsymmetric space. Since $\mathfrak{g}$ is semisimple we can write it as a direct sum $\mathfrak{g}=\mathfrak{g}_1\oplus\mathfrak{g}_2$ and factor our homomorphism $\rho$ as $\mathfrak{g}_1\oplus\mathfrak{g}_2\stackrel{\pi_2}{\rightarrow} \mathfrak{g}_2 \stackrel{\tilde{\rho}}{\longrightarrow} \mathfrak{g}^{'}$. As projections are tight, we get that $\rho$ is tight if and only if $\tilde{\rho}$ is tight.
\begin{lma}\label{lma3}
Let $\rho\colon \mathfrak{g}\rightarrow \mathfrak{g}^{'}$ be a representation of Hermitian Lie algebras respecting the Cartan decomposition and satisfying the condition (H1). Then there exist a regular subalgebra $\mathfrak{g}^{''}\subset\mathfrak{g}^{'}$ such that $\rho(\mathfrak{g})\subset\mathfrak{g}^{''}$ and $\rho\colon \mathfrak{g}\rightarrow \mathfrak{g}^{''}$ satisfies (H2).
Further, if $\mathfrak{g}$ is a sum of simple Hermitian Lie algebras $\mathfrak{g}_1\oplus...\oplus\mathfrak{g}_n$ we can choose  $\mathfrak{g}^{''}$ as a sum of Hermitian Lie algebras $\mathfrak{g}_1^{''}\oplus ... \oplus \mathfrak{g}_n^{''}$ such that the image of each $\mathfrak{g}_i$ is contained in $\mathfrak{g}^{''}_i$. The restricted maps $\rho|\colon \mathfrak{g}_i\rightarrow \mathfrak{g}_i^{''}$ all satisfy (H2).
\end{lma}
The inclusion map of regular subalgebras satisfy (H1) \cite{A6}. Lemma \ref{lma3} is one of the key reductions in classifying holomorphic maps between Hermitian symmetric spaces. It breaks down the classifying into two managable parts, classifying (H2)-representations and regular subalgebras. Together with the Lemma \ref{lma2} it allows us to break down our classification into the same parts, classifying tight (H2)-representations $\rho\colon \mathfrak{g}\rightarrow\mathfrak{g}^{''}$ and tight regular subalgebras $\mathfrak{g}^{''}\subset\mathfrak{g}^{'}$. 
The second part of Lemma \ref{lma3} allows us to reduce our classification to the case of $\mathfrak{g}$ simple.

We will now recall some facts about continuous bounded cohomology. We will restrict ourselves to cohomology in the second degree and to Hermitian Lie groups only. For a more thorough review of the theory we recommend \cite{A10}, and for connections to rigidity questions \cite{A11}.

Let $G$ be a Hermitian Lie group, i.e one which is the isometry group of a Hermitian symmetric space. Further let $G=G_1\times...\times G_n$ be a decomposition of $G$ into simple factors and $\mathcal{X}=\mathcal{X}_1\times... \times\mathcal{X}_n$ be the corresponding decomposition of the symmetric space into irreducible symmetric spaces. Let $\pi_i\colon G\rightarrow G_i$ be the projection maps and $\iota_i\colon G_i\rightarrow G$ be inclusion maps. We have
\begin{equation}\label{isom}
H_{cb}^2(G)\cong \prod{H^2_{cb}(G_i)}\cong \prod{\mathbb{R}\kappa_{G_i}^b}
\end{equation}
where the $\kappa_{G_i}^b$ are the classes we defined in section \ref{descr}.
The isomorphism is $\kappa\mapsto (\iota_i^{*}\kappa)_{i=1}^n$ with inverse $(\kappa_i)_{i=1}^n\mapsto\sum_{i=1}^n{\pi_i^{*}\kappa_i}$.
Under this isomorphism we have
$$\kappa_{G}^b=\sum_i{\kappa_{G_i}^b}.$$
The seminorm on $H^2_{cb}(G)$ is in fact a norm \cite{A11}.

For a class $\kappa\in H^2_{cb}(G)$ we can with a slight abuse of notation write $\kappa=\sum_{i}{\lambda_1\kappa_{G_i}^b}$ and we have
\begin{eqnarray*}
||\kappa||&=&sup_{(g_0,g_1,g_2)\in G}\kappa(g_0,g_1,g_2)=sup_{(g_0,g_1,g_2)\in G}\sum_i{\lambda_i\kappa_{G_i}^b(\pi_ig_0,\pi_ig_1,\pi_ig_2)}\\
&=&\sum_i{sup_{(g_0,g_1,g_2)\in G_i}\lambda_i\kappa_{G_i}^b(g_0,g_1,g_2)}=\sum_i{|\lambda_i| ||\kappa_{G_i}^b||}
\end{eqnarray*}
This implies that for positive classes $\kappa_1,\kappa_2\in H^2_{cb}(G)$ we have $||\kappa_1 + \kappa_2||=||\kappa_1|| +|| \kappa_2||$.

Fixing a $G$-invariant complex structure on $\mathcal{X}$ determines $G_i$-invariant complex structures on $\mathcal{X}_i$. These in turn determines a choice of K\"{a}hler forms $\omega_i$ which in turn determines a choice of basis $\kappa_{G_i}^b$ and hence an orientation on the $H^2_{cb}(G_i)$:s. In the following we assume such a choice of complex structures is made.
\begin{df}
We say that a class $\alpha \in H^2_{cb}(G)$ is
\begin{enumerate}
       \item  positive if $\alpha=\sum{\mu_i \kappa_{G_i}^b}$ where $\mu_i\geq 0$ for all $i=1,...,n$, and
       \item  strictly positive if $\alpha=\sum{\mu_i \kappa_{G_i}^b}$ where $\mu_i> 0$ for all $i=1,...,n$.
   \end{enumerate}
\end{df}
\begin{df}
We say that a homomorfism $\rho\colon H\rightarrow G$ is (strictly) positive if $\rho^* \kappa_{G}^b$ is (strictly) positive.
\end{df}
\begin{lma}\label{lma5}
If $\rho\colon G_1\rightarrow G_2$ is holomorphic then it is positive. If it is holomorphic and injective then it is strictly positive.
\end{lma}
\begin{proof}\ref{lma5}
Let $\mathcal{X}_i$ denote the corresponding symmetric spaces with K\"{a}lher forms $\omega_i$, complex structures $J_i$ and Riemanian metrics $g_i$ for $i=1,2$. We denote the induced map between the symmetric spaces by $\rho$ also. $\rho^*\omega_2(X,Y)=g_2(\rho_*X, J_2\rho_*Y)=g_2(\rho_*X, \rho_*J_1Y)$. Since $\rho$ is an isometry up to scaling we have that  $\rho^*\omega_2(X,Y)$ is a positive multiple  of $\omega_1(X,Y)$ hence $\rho$ is positive. If we further assume injectivity we can conclude that this multiple is non-zero hence $\rho$ is strictly positive.
\end{proof}

\begin{proof}\ref{lma1}
 Let $H$ be the isometry group of $\mathcal{Y}$ and $G=G_1\times ... \times G_n$ the isometry group of $\mathcal{X}_1\times ... \times \mathcal{X}_n$. We write $\rho_i$ as $\rho=(\rho_1,...,\rho_n)$. We have $$||\rho^*\kappa^b_G||=||\sum{\rho_i^* \kappa_{G_i}^b}||=\sum{||\rho_i^* \kappa_{G_i}^b||}\leq\sum{||\kappa_{G_i}^b||=||\kappa^b_{G}||}$$ where the second equality follows from positivity. Thus $\rho$ is tight if and only if all the $\rho_i$ are tight. 
\end{proof}
\begin{proof}\ref{lma2}
Using Lemma \ref{lma1} we may assume $\mathcal{Z}$ is irreducible. Let $\mathcal{Y}=\mathcal{Y}_1\times ... \times\mathcal{Y}_n$ be a decomposition of $\mathcal{Y}$ into irreducible parts. Let $ H,G=G_1\times...\times G_n , L$ denote the isometry groups of $\mathcal{X},\mathcal{Y},\mathcal{Z}$. Again we write $\rho$ as $\rho=(\rho_1,...,\rho_n)$. Assume that $\rho$ and $\eta$ are tight. 

We have $\eta^*\kappa^b_{L}=\sum{c_i \kappa^b_{G_i}}$ such that $ \sum{||c_i \kappa^b_{G_i}||}=\sum{c_i|| \kappa^b_{G_i}||}=||\kappa^b_{L}||$. We get
$||\rho^*\eta^*\kappa^b_{L}|| =||\sum{\rho_i^*c_i \kappa^b_{G_i}}||=\sum{c_i ||\rho_i^* \kappa^b_{G_i}||}=\sum{|c_i| || \kappa^b_{G_i}||}=||\kappa^b_{L}||$ so $\eta\circ\rho$ is tight. 

Assume instead that $\eta\circ\rho$ is tight. We have $||\kappa_L^b||=||\rho^{*}\eta^{*}\kappa_L^b||\leq||\eta^{*}\kappa_L^b||\leq||\kappa_L^b||$ which implies that $\eta$ is tight.
We have that $\eta^{*}\kappa_L^b=\sum{c_i\kappa^b_{G_i}}$ such that $\sum{c_i ||\kappa^b_{G_i}||}=||\kappa^b_{L}||$. Lemma \ref{lma5} assures us that all $c_i>0$.
The composition is tight so
$||\kappa^b_{L}|| =||\rho^*\eta^*\kappa^b_{L}|| =||\sum{\rho_i^*c_i \kappa^b_{G_i}}||=\sum{c_i ||\rho_i^* \kappa^b_{G_i}||}\leq\sum{c_i || \kappa^b_{G_i}||}=||\kappa^b_{L}||$ 
, so the inequality in the equation above is an equality.
Since we have established that all $c_i>0$ the inequality is an equality if and only if $||\rho_i^{*}\kappa^b_{G_i}||=||\kappa_{G_i}^b||$ for all $i$ which is true if and only $\rho$ is tight.
\end{proof}
The norms $||\kappa^b_{G}||$ were computed in \cite{A4} and equals $r_{\mathcal{X}}\pi$, where $r_{\mathcal{X}}$ is the rank of the symmetric space $\mathcal{X}$ associated to $G$. Another approach using the Maslov index can be found in \cite{A12}.

Regular subalgebras are formed by choosing a subset of the root system. Fulfilling certain conditions this subset works as the set of simple roots for the subalgebra. More precisely, let $\mathfrak{g}$ be a Hermitian Lie algebra with complexification $\mathfrak{g}^{\mathbb{C}}$. Let $\Lambda$ denote the root system of $\mathfrak{g}^{\mathbb{C}}$. A subset $\Delta$ of $\Lambda$ is called a $\Pi$-system if it satisfies\\
(i) If $\alpha,\beta \in \Delta$ then $\alpha-\beta\notin\Delta$\\
(ii) $\Delta$ is a linearly independent set in $i\mathfrak{h}^*$\\
(iii) Each connected component of the Dynkin diagram of $\Delta$ contains at most one non-compact root.\\
From this $\Pi$-system we form a new root system 
$\Lambda(\Delta):=(\sum_{\alpha\in\Delta}{\mathbb{Z}\alpha)\cap\Lambda}$\\ and a subalgebra of the complexification by
$\mathfrak{g}^{\mathbb{C}}(\Delta):=\sum_{\alpha\in\Delta}{\mathbb{C}H_\alpha}+\sum_{\alpha\in\Lambda(\Delta)}{\mathfrak{g}^\alpha}$.\\
Finally we define the regular subalgebra as $\mathfrak{g}(\Delta):=\mathfrak{g}^{\mathbb{C}}(\Delta)\cap\mathfrak{g}$.

\section{Criterions for tightness}

\begin{thm}\label{subalg}
Let $\mathfrak{g}$ be a simple Hermitian Lie algebra with regular subalgebra $\mathfrak{g}_1\oplus...\oplus\mathfrak{g}_k$, each $\mathfrak{g}_i$ being simple, and  $\mathcal{X}_1\times...\times\mathcal{X}_k \subset \mathcal{X}$ the corresponding symmetric spaces. Let $\Lambda$ be the root system of $\mathfrak{g}^{\mathbb{C}}$ and $\Lambda_i$ the sub-root systems corresponding to the $\mathfrak{g}_i^{\mathbb{C}}$. Let $\gamma_i$ be the highest root of $\Lambda_i$ with respect to some ordering of $\Lambda_i$ and let $\gamma$ be the highest root of $\Lambda$. Further let $r$ and $r_i$ denote the ranks.
If 
\begin{equation}
\sum_{i}{  \frac{\langle\gamma,\gamma\rangle}{\langle\gamma_i,\gamma_i\rangle}r_i  }=r,\label{kvoten}
\end{equation}
where the brackets denote the Killing form of $\mathfrak{g}^{\mathbb{C}}$, then the inclusion  $\mathcal{X}_1\times...\times\mathcal{X}_k\subset \mathcal{X}$ is tight.
\end{thm}

\begin{proof}
Let $g$ and $g_i$ be the invariant metrics on $\mathcal{X}$ respectively $\mathcal{X}_i$, normalized so that the minimal holomorphic sectional curvature is $-1$. Further let $\omega$ and $\omega_i$ be the associated K\"{a}hler forms. We have that $g|_{\mathcal{X}_i}=c_i g_i$ for some $c_i$ and since the inclusion is holomorphic we also have $\omega|_{\mathcal{X}_i}=c_i\omega_i$. Here $c_i\geq 1$ since the holomorphic sectional curvature of $g$ is $\geq -1$ on $\mathcal{X}_i$.
 We get 
\begin{eqnarray}
sup_{\Delta\subset\mathcal{X}_1\times...\times \mathcal{X}_k } \int_{\Delta}\omega&=&\sum_i{sup_{\Delta\subset\mathcal{X}_i}\int_{\Delta}\omega}\label{kurv1}\\
&=&\sum_i{sup_{\Delta\subset\mathcal{X}_i}\int_{\Delta}c_i \omega_i}=\sum_i{c_i r_i \pi}\nonumber
\end{eqnarray}
thus the inclusion is tight if $\sum_i{c_i r_i}=r$. Let us examine the $c_i$ further.
The minimal holomorphic sectional curvature on $\mathcal{X}_i$ with respect to $g$, $mincurv(\mathcal{X}_i,g)$, is $-\frac{1}{c_i}$.
If we consider the quotient $\frac{mincurv(\mathcal{X},g)}{mincurv(\mathcal{X}_i,g)}=c_i$ we can calculate the $c_i$ without concerns about the normalization of the metric.
Identifying  $T_0\mathcal{X}$ with $\mathfrak{p}$ with for some reference point $0$, we have $g(X,Y)= b\langle X,Y\rangle $ for some $b>0$ where the brackets denote the Killing form of $\mathfrak{g}^{\mathbb{C}}$. 
Let $\mathfrak{g}^{\mathbb{C}}$ be the complexification of $\mathfrak{g}$, and $\mathfrak{g}^{\mathbb{C}}=\mathfrak{h}+\sum_\alpha{\mathfrak{g}^\alpha}$ a root space decomposition.
Pick $e\in\mathfrak{g}^{\gamma}$ and $\bar{e}\in\mathfrak{g}^{-\gamma}$  such that $\xi=e+\bar{e}\in \mathfrak{p}$.  We reach the minimal holomorphic sectional curvature in the complex line spanned by $\xi$. This can be deduced from the construction by Harish-Chandra of the realisation of $\mathcal{X}$ as a bounded symmetric domain, see for example \cite{A3}. Further let $H_\gamma$ be the element of $\mathfrak{h}$ such that $\langle H,H_\gamma\rangle =\gamma(H)$ for all $H\in\mathfrak{h}$.  We get 
\begin{eqnarray}
mincurv(\mathcal{X},g)&=&\frac{b\langle[\xi,J\xi],[\xi,J\xi]\rangle }{b^2\langle\xi,\xi\rangle \langle J\xi,J\xi\rangle }=\frac{\langle[\xi,J\xi],[\xi,J\xi]\rangle }{b\langle\xi,\xi\rangle ^2}\nonumber\\
&=&\frac{\langle[e+\bar{e},ie-i\bar{e}],[e+\bar{e},ie-i\bar{e}]\rangle }{b\langle e+\bar{e},e+\bar{e}\rangle ^2}=\frac{4\langle[e,\bar{e}],[e,\bar{e}]\rangle }{4b\langle e,\bar{e}\rangle ^2}\\
&=&\frac{\langle\langle e,\bar{e}\rangle H_\gamma,\langle e,\bar{e}\rangle H_\gamma\rangle }{b\langle e,\bar{e}\rangle ^2}=\frac{\langle \gamma,\gamma\rangle }{b}
\end{eqnarray}
Now let $\gamma_i$ denote the highest root of $\mathfrak{g}_i^{\mathbb{C}}$ and pick $\bar{e_i},e_i$ like with $\gamma$. We then get 
$$c_i=\frac{mincurv(\mathcal{X},g)}{mincurv(\mathcal{X}_i,g)}=\frac{\langle \gamma,\gamma\rangle }{\langle \gamma_i,\gamma_i\rangle }.$$
\end{proof}

\begin{cor}\label{cor}
Let $\rho\colon\mathcal{X}_1\times...\times\mathcal{X}_n\rightarrow \mathcal{X}$ be a holomorphic, totally geodesic embedding. Let $r_i$ be the rank of $\mathcal{X}_i$ and $r$ the rank of $\mathcal{X}$. If $\sum{r_i}=r$ then $\rho$ is tight.
\end{cor}

\begin{proof}
Viewing $\mathcal{X}_1\times ... \times\mathcal{X}_n$ as a subsymmetric space of $\mathcal{X}$ we can argue as in the proof of Theorem \ref{subalg}. From (\ref{kurv1}) we have that the inclusion is tight if $\sum{c_i r_i}=r$. We also have that $c_i\geq 1$, $\sum{c_i r_i}\leq r$ and that $\sum{r_i}=r$ by assumption.
Thus all the $c_i=1$ and the inequality is an equality, i.e $\rho$ is tight.
\end{proof}

As is well known, we can in each Hermitian symmetric space $\mathcal{X}$ of rank $r$ holomorphically embed the product of  $r$ Poincare discs $\mathbb{D}$. This embedding is tight\cite{A8}. Lemma \ref{lma1} also shows that the diagonal embedding $d\colon \mathbb{D}\rightarrow \mathbb{D}\times...\times\mathbb{D}$ is tight. The composition of $d\colon \mathbb{D}\rightarrow \mathbb{D}^{r}$ and $\iota\colon  \mathbb{D}^{r}\rightarrow \mathcal{X}$ is then a tight embedding. We call these embeddings \emph{diagonal discs}. They will be needed in the first of the following two Theorems from \cite{A8}.

\begin{thm}\label{diagd}
Let $\rho\colon\mathcal{X}\rightarrow \mathcal{X}^{'}$ be a holomorphic and totally geodesic embedding. Further let $d$ and $d^{'}$ be embeddings of diagonal discs in $\mathcal{X}$ and $\mathcal{X}^{'}$. Denote by the same letters the corresponding Lie algebra homomorphisms, where we let  $\mathfrak{g}$ and $\mathfrak{g}^{'}$ denote the Lie algebras corresponding to $\mathcal{X}$ and $\mathcal{X}^{'}$. Further let $Z_{\mathbb{D}}\in su(1,1)$ and $Z^{'}\in\mathfrak{g}^{'}$ be the complex structures of the disc and of $\mathcal{X}^{'}$. Then $\rho$ is tight if and only if $\langle \rho dZ_{\mathbb{D}},Z^{'}\rangle =\langle d^{'}Z_{\mathbb{D}},Z^{'}\rangle $, where the brackets denote the Killing form of $\mathfrak{g}^{'}$.
\end{thm}
\begin{thm}\label{tube}
Suppose $\mathcal{X}_1$, $\mathcal{X}_2$ are Hermitian symmetric spaces of tube type. Then $\rho\colon \mathcal{X}_1\rightarrow\mathcal{X}_2$ is a tight and holomorphic embedding if and only if the corresponding Lie algebra homomorphism $\rho\colon\mathfrak{g}_1\rightarrow\mathfrak{g}_2$ is an (H2)-homomorphism.
\end{thm}
We now have the tools we need to determine if a representation corresponds to a tight map.

\section{regular subalgebras}
In this section we examine the regular subalgebras to see which are tight. To do this we will use Theorem \ref{subalg} and Corollary \ref{cor}. Since the criterion in Theorem \ref{subalg} is a quotient of norms we can normalize the Killing form as we please. We can thus pretty much read off the results directly from the Dynkin diagrams. To make calculations easy we normalize so that the length of the shortest roots are 1. We recall some basic facts about root systems that will make the calculations easier.
There are at most two lengths of roots for each simple Hermitian Lie algebra. In fact, for the Lie algebras we are concerned with, only $sp(2p)$ and $so(2l+1,2)$ contains roots of different lengths. With the normalization we have chosen the longer roots of these Lie algebras have length $\sqrt{2}$. Also, the highest root of a root system is always a longer root.

We divide this section into subsections, one for each simple Hermitian Lie algebra $\mathfrak{g}^{'}$. In each subsection $\{\alpha_i\}$ denotes the set of simple roots and $\gamma$ the highest root of $\mathfrak{g}^{'}$. We denote the subalgebras by $\mathfrak{g}$ and their decomposition into simple parts by $\mathfrak{g}=\mathfrak{g}_1\oplus...\oplus\mathfrak{g}_n$. For each $\mathfrak{g}_i$ we denote the highest root by $\gamma_i$. We begin each subsection by describing the Dynkin diagram and listing roots that will be needed to describe the subalgebras. We then list the systems of roots corresponding to the maximal regular subalgebras as classified by Ihara \cite{A6}. We then calculate which of them are tightly embedded in the original Lie algebra.
We list only the proper maximal regular subalgebras but we cover in fact all regular subalgebras. We do this by considering chains of proper maximal regular subalgebras. If $\mathfrak{g}$ is a non-maximal regular subalgebra of $\mathfrak{g}^{'}$, we can find a proper maximal regular subalgebra $\mathfrak{g}_1$ containing $\mathfrak{g}$. If $\mathfrak{g}$ is a proper maximal regular subalgebra of $\mathfrak{g}_1$ we are done, otherwise we repeat the process until we have a chain $\mathfrak{g}\subset\mathfrak{g}_k\subset ... \subset\mathfrak{g}_1\subset\mathfrak{g}^{'}$ of proper maximal regular subalgebras. We have that $\mathfrak{g}$ is a tight regular subalgebra of $\mathfrak{g}^{'}$ if all the inclusions in the chain are tight.

Our list looks a little different than Ihara's as we are only interested in the part of the subalgebras corresponding to a Hermitian symmetric spaces of non-compact type. 
To avoid overlappings and special cases of the Dynkin diagrams in lower dimensions  we restrict param eters as follows:\\
 $su(p,q)$, $1\leq p\leq q$, $so^*(2p)$, $p\geq 5$,
$sp(2p)$, $p\geq 2$, $so(p,2)$ $p\geq 5$\\
All cases are still covered due to wellknown isomorphisms, see for example \cite{A3}.\\

{\bf Maximal regular subalgebras $\mathfrak{g}$ of $\mathfrak{g}^{'}=su(p,q)$}\\

\begin{picture}(50,43)
\multiput(5,31)(30,0){7}{\circle{5}}
\multiputlist(20,31)(30,0)%
{{\line(10,0){25}},{$\cdots$},{\line(10,0){25}},{\line(10,0){25}},{$\cdots$},{\line(10,0){25}}}
\multiputlist(5,41)(30,0){$\alpha_{q+1}$,$\alpha_{q+2}$,$\alpha_{q+p-1}$,$\alpha_{1}$,$\alpha_{2}$,$\alpha_{q-1}$,$\alpha_{q}$}
\end{picture}
\begin{align}
&\gamma=\alpha_1+...+\alpha_{p+q-1}\label{gammat}
\end{align}
\begin{align}
&su(l,q) ,\,1\leq l <p,\label{su1}\\
& \{\alpha_{p+q-l},\,\alpha_{p+q-l+1},...,\alpha_{p+q-1},\alpha_1,\alpha_2,...,\alpha_q\}\nonumber\\
&su(p,s) ,\,p\leq s <q,  \label{su2}\\
&\{\alpha_{q+1},\,\alpha_{q+2},...,\alpha_{p+q-1},\alpha_1,\alpha_2,...,\alpha_s\}\nonumber\\
&su(s,p),\, 1\leq s<p, \, \label{su3}\\
&\{\alpha_{s},\alpha_{s-1},...,\alpha_{2},\alpha_1,\alpha_{p+q-1},\alpha_{p+q-2},...,\alpha_{q+1}\}\nonumber
\end{align}
\begin{align}
&su(l,s)+su(p-l,q-s),\, 1\leq l\leq s, \, p-l\leq q-s, \,   \label{su4}\\
&\begin{array}{l}  \{\alpha_{p+q-l},\alpha_{p+q-l+1},...,\alpha_{p+q-2},\alpha_1,\alpha_2,...,\alpha_s\}\cup\\
 \{-\alpha_{p+q-l-2},...,-\alpha_{q+1},\gamma,-\alpha_q,...,-\alpha_{s+2}\}
\end{array}\nonumber\\
&su(s,l)+su(p-l,q-s),\, 1\leq s<l <p, \,\label{su5}\\
&\begin{array}{l}  \{\alpha_{s},\alpha_{s-1},...,\alpha_1,\alpha_{p+q-1},...,\alpha_{p+q-l}\}\cup\\
 \{-\alpha_{p+q-l-2},...,-\alpha_{q+1},\gamma,-\alpha_q,...,-\alpha_{s+2}\}\nonumber
\end{array}
\end{align}

Since all roots in $su(p,q)$ are of the same length, the quotient $\frac{\langle\gamma,\gamma\rangle}{\langle\gamma_i,\gamma_i\rangle}$ in (\ref{kvoten}), Theorem \ref{subalg} is 1 in all cases. Comparing ranks we see that (\ref{su2}) and (\ref{su4}) are tight, the rest are not.\\

{\bf Maximal regular subalgebras $\mathfrak{g}$ of $\mathfrak{g}^{'}=so^*(2p)$}\\

\begin{picture}(50,43)
\multiput(5,31)(30,0){5}{\circle{5}}
\multiputlist(20,31)(30,0)%
{{\line(10,0){25}},{\line(10,0){25}},{$\cdots$},{\line(10,0){25}}}
\multiputlist(5,41)(30,0){$\alpha_{1}$,$\alpha_{2}$,$\alpha_{3}$,$\alpha_{p-2}$,$\alpha_{p-1}$}
\put(35,0){\circle{5}}
\put(34,4){\line(0,10){25}}
\put(20,0){$\alpha_p$}
\end{picture}
\begin{align}
&\gamma=\alpha_1+2(\alpha_2+...+\alpha_{p-2})+\alpha_{p-1}+\alpha_p \\
&\beta=\alpha_2+ \alpha_3+...+\alpha_p
\end{align}
\begin{align}
&su(l,p-l) ,\, 1\leq l\leq p/2\, ,\label{so1}\\
&\{-\alpha_{p-l+2},...,-\alpha_{p-2},-\alpha_{p-1},\beta,\alpha_1,...,\alpha_{p-l}\}\nonumber\\
&so^*(2l) + so^*(2(p-l)) ,\, [p/2]\leq l \leq p-2  \, ,  \label{so2}\\
&\Big\{\begin{array}{l}
\alpha_1,...,\alpha_{l-1}\\
\alpha_p
\end{array}\Big\}
\cup 
\Big\{\begin{array}{l}
\gamma,-\alpha_{p-2},-\alpha_{p-3}...,-\alpha_{l+1}\\
-\alpha_{p-1}
\end{array}\Big\}\nonumber\\
&so^*(2(p-1)), \,
\Big\{\begin{array}{l}
 \alpha_1,...,\alpha_{p-2}\\
\alpha_p
\end{array}\Big\}\label{so3}
\end{align}

Again all roots are of the same length. By comparing ranks we see that (\ref{so1}) is tight  precisely when $l=[\frac{p}{2}]$.\\
For (\ref{so2}) we see that the inclusion is tight for all $l$ if $p$ is odd, and for even $l$ if $p$ is even. For (\ref{so3}) the inclusion is tight for $p$ odd.\\

{\bf Maximal regular subalgebras $\mathfrak{g}$ of $\mathfrak{g}^{'}=sp(2p)$}\\

\begin{picture}(50,23)
\multiput(5,1)(30,0){5}{\circle{5}}
\multiputlist(50,1)(30,0)%
{{\line(10,0){25}},{$\cdots$},{\line(10,0){25}}}
\multiputlist(5,11)(30,0){$\alpha_{1}$,$\alpha_{2}$,$\alpha_{3}$,$\alpha_{p-1}$,$\alpha_{p}$}
\put(7,2){\line(10,0){25}}
\put(7,0){\line(10,0){25}}
\end{picture}

\begin{align}
&\gamma=\alpha_1+2\sum_2^{p}{\alpha_i} \\
&\beta=\sum_2^{p}{\alpha_i}
\end{align}
\begin{align}
&su(l,p-l) \:1\leq l\leq [p/2],\label{sp1}\\
& \{-\alpha_{p-l+3},...,-\alpha_{p-1},-\alpha_p, \beta,\alpha_1+\alpha_2,\alpha_3,\alpha_4,...,\alpha_{p-l+1}\}\nonumber\\
&sp(2l) + sp(2(p-l)) \:[p/2]\leq l \leq p-1   , \label{sp2}\\
&\{\alpha_{1},...,\alpha_{l}\}\cup\{\gamma,-\alpha_p,\alpha_{p-1},...,-\alpha_{l+2}\}\nonumber
\end{align}
Here there are two possible lengths for roots so we have to make some calculations.
For (\ref{sp1}) we use the formula (\ref{gammat}) for the highest root of $\mathfrak{g}$ of type $su(p,q)$. We get  $$\gamma_1=-\sum_{p-l+3}^{p}{\alpha_i}+\sum_2^{p}{\alpha_i}+\sum_1^{p-l+1}{\alpha_i}=\alpha_1+2\sum_2^{p-l+1}{\alpha_i}+\alpha_{p-l+2}$$ as the highest root.
We have
\begin{align*}
\langle \gamma_1,\gamma_1\rangle = ||\alpha_1||^2+||2\sum_2^{p-l+1}{\alpha_i}||^2+||\alpha_{p-l+2}||^2+2\langle \alpha_1,2\alpha_2\rangle \\ +2\langle 2\alpha_{p-l+1},\alpha_{p-l+2}\rangle   =2+4+1-4-2=1
\end{align*}
and  $\langle \gamma,\gamma\rangle =2$ since it is the highest root.
By Theorem \ref{subalg} the inclusion is tight if l=p/2.
Using Corollary \ref{cor} we see that (\ref{sp2}) is tight .\\

{\bf Maximal regular subalgebras $\mathfrak{g}$ of $\mathfrak{g}^{'}=so(p,2)$,  $p= 2k-2$ even} \\

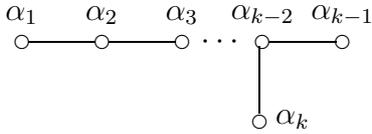
\begin{picture}(50,37)
\multiput(5,31)(30,0){5}{\circle{5}}
\multiputlist(20,31)(30,0)%
{{\line(10,0){25}},{\line(10,0){25}},{$\cdots$},{\line(10,0){25}}}
\multiputlist(5,41)(30,0){$\alpha_{1}$,$\alpha_{2}$,$\alpha_{3}$,$\alpha_{k-2}$,$\alpha_{k-1}$}
\put(94,1){\circle{5}}
\put(94,4){\line(0,10){25}}
\put(100,0){$\alpha_k$}
\end{picture}
\begin{align}
&\gamma=\alpha_1+2(\alpha_2+...+\alpha_{k-2})+\alpha_{k-1}+\alpha_k\\
&\beta_1=\alpha_2+2(\alpha_3+...+\alpha_{k-2})+\alpha_{k-1}+\alpha_k\\
&\beta_2=\alpha_{k-2}+\alpha_{k-1}+\alpha_k
\end{align}
\begin{align}
&su(1,1)+ su(1,1),\, \{\alpha_1\}\cup\{\gamma\}\label{soj1}\\
&su(1,l),\, 2\leq l \leq k-1, \,\{\alpha_1,...,\alpha_l\}\label{soj2}\\
&su(1,k-1), \, \{\alpha_1,...,\alpha_{k-2},\alpha_k\}\label{soj3}\\
&su_(2,2),\, \{\beta_1,\alpha_1,\alpha_2\}\label{soj4}\\
&so(p-2,2),\, \Big\{\begin{array}{l}
\alpha_2,...,\alpha_{k-1}\\
\beta_2
\end{array}\Big\}\label{soj5}
\end{align}

Again we have that all roots have the same length. Comparing ranks, Theorem \ref{subalg} then tells us that (\ref{soj1}), (\ref{soj4}) and (\ref{soj5}) are tight inclusions, the rest are not.\\

\newpage
{\bf Maximal regular subalgebras $\mathfrak{g}$ of $\mathfrak{g}^{'}=so(p,2)$, $p=2k-1$ odd}\\

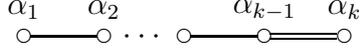
\begin{picture}(50,57)
\multiput(5,31)(30,0){5}{\circle{5}}
\multiputlist(20,31)(30,0)%
{{\line(10,0){25}},{$\cdots$},{\line(10,0){25}}}
\multiputlist(5,41)(30,0){$\alpha_{1}$,$\alpha_{2}$,,$\alpha_{k-1}$,$\alpha_{k}$}
\put(97,32){\line(10,0){25}}
\put(97,30){\line(10,0){25}}
\end{picture}
\begin{align}
&\gamma=\alpha_1+2(\alpha_2+...+\alpha_{k-2})+\alpha_{k-1}+\alpha_k\\
&\beta_1=\alpha_2+2(\alpha_3+...+\alpha_{k-1}+\alpha_k)\\
&\beta_2=\alpha_{k-1}+2\alpha_k\\
&\beta_3=\alpha_{k-1}+\alpha_k\\
&\beta_4=\alpha_1+\alpha_2+...+\alpha_{k-2}+\alpha_{k-1}+\alpha_k
\end{align}
\begin{align}
&su(1,1)+su(1,1),\, \{\alpha_1\}\cup\{\gamma\}\label{sou1}\\
&su(1,l),\, 2\leq l \leq k-2, \,\{\alpha_1,...,\alpha_l\}\label{sou2}\\
&su(2,2),\, \{\beta_1,\alpha_1,\alpha_2\}\label{sou3}\\
&su(1,1), \, \{\beta_4\}\label{sou4}\\
&so(p-1,2),\, \Big\{\begin{array}{l}
\alpha_2,...,\alpha_{k-1}\\
\beta_2
\end{array}\Big\}\label{sou5}\\
&so(p-2,2),\, \{\alpha_2,...,\alpha_{k-1},\beta_3\}\label{sou6}
\end{align}
Here (\ref{sou1}), (\ref{sou3}), (\ref{sou5}) and (\ref{sou6}) are tight by comparing ranks and applying Corollary \ref{cor}. Calculating the quotient for (\ref{sou4}) gives us $c_1=2$, the inclusion is thus tight by Theorem \ref{subalg}. For (\ref{sou2}) the quotient is one and the inclusion is not tight.

{\bf Maximal regular subalgebras $\mathfrak{g}$ of $\mathfrak{g}^{'}=e_{6(-14)}$}\\

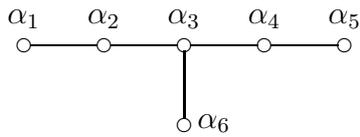
\begin{picture}(50,37)
\multiput(5,31)(30,0){5}{\circle{5}}
\multiputlist(20,31)(30,0)%
{{\line(10,0){25}},{\line(10,0){25}},{\line(10,0){25}},{\line(10,0){25}}}
\multiputlist(5,41)(30,0){$\alpha_{1}$,$\alpha_{2}$,$\alpha_{3}$,$\alpha_{4}$,$\alpha_{5}$}
\put(65,1){\circle{5}}
\put(65,4){\line(0,10){25}}
\put(70,0){$\alpha_6$}
\end{picture}
\begin{align}
&\gamma=\alpha_1+2\alpha_2+3\alpha_3+2\alpha_4+\alpha_5+2\alpha_6\\
&\beta_1=\alpha_2+2\alpha_3+2\alpha_4+\alpha_5+\alpha_6\\
&\beta_2=\alpha_3+\alpha_4+\alpha_5+\alpha_6
\end{align}
\begin{align}
&su(1,5)+su(1,1), \, \{\alpha_1,...,\alpha_5\}\cup\{\gamma\}\\
&su(1,2)+su(1,2), \, \{\alpha_1,\alpha_2\}\cup\{\gamma,-\alpha_6\}\\
&su(2,4),\, \{\beta_1,\alpha_1,\alpha_2,\alpha_3,\alpha_6\}\\
&so^*(10),\, \Big\{\begin{array}{l}
\alpha_1,\alpha_2,\alpha_3,\alpha_4\\
\beta_2
\end{array}\Big\}\\
&so(8,2),\, \Big\{\begin{array}{l}
\alpha_1,\alpha_2,\alpha_3,\alpha_4\\
\alpha_6
\end{array}\Big\}
\end{align}
We see immeadiately that all the inclusions are tight by comparing ranks and applying Corollary \ref{cor}.

{\bf Maximal regular subalgebras $\mathfrak{g}$ of $\mathfrak{g}^{'}=e_{7(-25)}$}\\

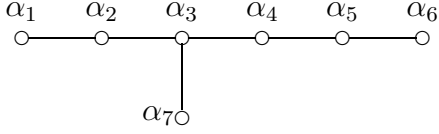
\begin{picture}(50,37)
\multiput(5,31)(30,0){6}{\circle{5}}
\multiputlist(20,31)(30,0)%
{{\line(10,0){25}},{\line(10,0){25}},{\line(10,0){25}},{\line(10,0){25}},{\line(10,0){25}}}
\multiputlist(5,41)(30,0){$\alpha_{1}$,$\alpha_{2}$,$\alpha_{3}$,$\alpha_{4}$,$\alpha_{5}$,$\alpha_{6}$}
\put(65,1){\circle{5}}
\put(65,4){\line(0,10){25}}
\put(50,0){$\alpha_7$}
\end{picture}
\\
\\
\begin{align}
&\gamma=\alpha_1 + 2\alpha_2 + 3\alpha_3 + 4\alpha_4 + 3\alpha_5 + 2\alpha_6 + 2\alpha_7\\
&\beta_1=\alpha_2 + 2\alpha_3 +2 \alpha_5 + \alpha_6 + 2 \alpha_7\\
&\beta_2=\alpha_3 + 2\alpha_4 + 2\alpha_5 + \alpha_6 + \alpha_7\\
&\beta_3=\alpha_4 + \alpha_5 + \alpha_6 + \alpha_7  
\end{align}
\begin{align}
&su(1,5)+su(1,2), \, \{\alpha_1,...,\alpha_4,\alpha_7\}\cup\{\gamma,-\alpha_6\}\label{e1}\\
&su(1,3)+su(1,3), \, \{\alpha_1,\alpha_2,\alpha_3\}\cup\{\gamma,-\alpha_6\,-\alpha_5\}\label{e2}\\
&su(2,6), \, \{\beta_1,\alpha_1,\alpha_2,...,\alpha_6\}\label{e3}\\
&su(3,3), \, \{-\alpha_7,\beta_1,\alpha_1,\alpha_2,\alpha_3\}\label{e4}\\
&so^*(12),\, \label{e5}\Big\{\begin{array}{l}
\alpha_1,\alpha_2,\alpha_3,\alpha_4,\alpha_7\\
\beta_2
\end{array}\Big\}\\
&so(10,2)+su(1,1), \,
\Big\{\begin{array}{l}
\alpha_1,\alpha_2,\alpha_3,\alpha_4,\alpha_5\\
\alpha_7
\end{array}\Big\}
\cup\{\gamma\}\label{e6}\\
&e_{6(-14)},\, \Big\{\begin{array}{l}
\alpha_2,\alpha_3,\alpha_4,\alpha_5,\alpha_6\\
\beta_3\}
 \end{array}\Big\}\label{e7}
\end{align}

Again all the roots are of length 1. Comparing ranks and applying Theorem \ref{subalg} we have that (\ref{e4}), (\ref{e5}) and (\ref{e6}) are tight, the rest are not.

\section{Irreducible representations satisfying (H2)}

The remaining problem is to classify irreducible tight \linebreak (H2)-representations. Irreducible (H2)-representations $\rho\colon\mathfrak{g}\rightarrow\mathfrak{g}^{'}$ were classified by Satake and Ihara so we just have to check which of them are tight using Theorem \ref{diagd} and Corollary \ref{tube}. This will require some explicit calculations. We therefore begin by giving matrix descriptions of $su(p,q), sp(2p)$ and $so^*(2p)$ together with Cartan decompositions and elements defining a complex structure on the corresponding symmetric space. We also define two diagonal discs that will be needed in the calculations.\\

Let $V$ be a $(p+q)$-dimensional vector space over $\mathbb{C}$. $su(p,q)$ is defined as the traceless subalgebra of $End(V)$ preserving some sesquilinear form $F$ of signature $(p,q)$, i.e. $X\in su(p,q)$ if $F(Xv,w)+F(v,Xw)=0$ for all $v,w\in V$. Choosing an orthonormal basis $\{e_i\}$ for $F$ with $F(e_i,e_i)=1$ for $i\leq p$ we can represent $F$ with the matrix 
$\left(\begin{array}{cc}
1_p&0\\
0&-1_q
\end{array}\right) $, by $F(v,w)=v^* \left(\begin{array}{cc}
1_p&0\\
0&-1_q
\end{array}\right) w$. With respect to this basis we can identify $su(p,q)$ as the matrix algebra
\begin{eqnarray*}
\mathfrak{g}=\{ 
\left(\begin{array}{cc}
A&B\\
B^*&C
\end{array}\right) 
:A\in M_p(\mathbb{C}),B\in M_{p,q}(\mathbb{C}),C\in M_q(\mathbb{C}), A^*=-A, C^*=-C\}
\end{eqnarray*}
We choose the Cartan decomposition and complex structure 
\begin{eqnarray*}
\mathfrak{k}=\{ 
\left(\begin{array}{cc}
A&0\\
0&C
\end{array}\right) \}\, ,
\mathfrak{p}=\{ 
\left(\begin{array}{cc}
0&B\\
B^*&0
\end{array}\right) \}\, ,
Z_{p,q}=\frac{i}{p+q}\left(\begin{array}{cc}
q1_p&0\\
0&-p1_q
\end{array}\right) \\
\end{eqnarray*}

The algebra $so^*(2p)$ is defined using a symmetric $\mathbb{C}$-bilinear form and a skew-Hermitian form. Fixing a well chosen basis we can present $so^*(2p)$ as the traceless subalgebra of $End(\mathbb{C}^{2p})$ preserving the forms 
\begin{eqnarray}
F(v,w)&=&v^* 
\left(\begin{array}{cc}
i1_p&0\\
0&-i1_p
\end{array}\right) w
\mbox{ and}\label{soett}\\
Q(v,w)&=&v^t \left(\begin{array}{cc}
0&1_p\\
1_p&0
\end{array}\right) w.\label{sotva}
\end{eqnarray}
This gives us the following matrix description for $so^*(2p)$:
\begin{eqnarray*}
\mathfrak{g}=\{ 
\left(\begin{array}{cc}
A&B\\
B^*&\bar{A}
\end{array}\right) 
:A\in M_p(\mathbb{C}),B\in M_{p}(\mathbb{C}), A^*=-A, B^t=-B\}\\
\mathfrak{k}=\{ 
\left(\begin{array}{cc}
A&0\\
0&\bar{A}
\end{array}\right) \}\, ,
\mathfrak{p}=\{ 
\left(\begin{array}{cc}
0&B\\
B^*&0
\end{array}\right) \}\, ,
Z=\frac{i}{2}\left(\begin{array}{cc}
1_p&0\\
0&-1_p
\end{array}\right) 
\end{eqnarray*}

In a similar fashion $sp(2p)$ is defined as the algebra preserving a skewsymmetric bilinear form on $\mathbb{R}
^{2p}$. Conjugating this algebra in $End(\mathbb{C}^{2p})$ we arrive at the following description:

\begin{eqnarray*}
\mathfrak{g}=\{ 
\left(\begin{array}{cc}
A&B\\
B^*&\bar{A}
\end{array}\right) 
:A\in M_p(\mathbb{C}),B\in M_{p}(\mathbb{C}), A^*=-A, B^t=B\}\\
\mathfrak{k}=\{ 
\left(\begin{array}{cc}
A&0\\
0&\bar{A}
\end{array}\right) \}\, ,
\mathfrak{p}=\{ 
\left(\begin{array}{cc}
0&B\\
B^*&0
\end{array}\right) \}\, ,
Z=\frac{i}{2}\left(\begin{array}{cc}
1_p&0\\
0&-1_p
\end{array}\right) \\
\end{eqnarray*}

We also define the following two diagonal discs

\begin{align}
&d_{p,q}\colon su(1,1)\rightarrow su(p,q),\, p\geq q\\
&\left(\begin{array}{cc}
ai&z\\
\bar{z}&-ai
\end{array}\right)\mapsto
\left(\begin{array}{ccc}
0&0&0\\
0&ai1_q&z1_q\\
0&\bar{z}1_q&-ai1_q
\end{array}\right)\label{disk}
\end{align}
For $p=2l+1$  we define
\begin{align}
&d_{p}\colon su(1,1)\rightarrow so^*(2p)\\
&\left(\begin{array}{cc}
ai&z\\
\bar{z}&-ai
\end{array}\right)\mapsto
\left(\begin{array}{cccccc}
ai1_l&0&0&0&0&z1_l\\
0&0&0&0&0&0\\
0&0&ia1_l&-z1_l&0&0\\
0&0&-\bar{z}1_l&-ia1_l&0&0\\
0&0&0&0&0&0\\
\bar{z}1_l&0&0&0&0&-ai1_l
\end{array}\right).
\end{align}

From these descriptions we see that we have two inclusions
\begin{eqnarray}
\iota_1\colon sp(2p)\rightarrow su(p,p)\\
\iota_2\colon so^*(2p)\rightarrow su(p,p)
\end{eqnarray}
By Theorem \ref{tube} we have that $\iota_1$ is tight and that $\iota_2$ is tight for $p$ even. For $p=2l+1$ odd we calculate

\begin{align*}
&\iota_2 d_p Z_{1,1}-d_{p,p}Z_{1,1}=\\
&=\frac{i}{2}\left(\begin{array}{cccccc}
1_l&0&0&0&0&0\\
0&0&0&0&0&0\\
0&0&1_l&0&0&0\\
0&0&0&-1_l&0&0\\
0&0&0&0&0&0\\
0&0&0&0&0&-1_l
\end{array}\right)
-\frac{i}{2}
\left(\begin{array}{cccccc}
1_l&0&0&0&0&0\\
0&1&0&0&0&0\\
0&0&1_l&0&0&0\\
0&0&0&-1_l&0&0\\
0&0&0&0&-1&0\\
0&0&0&0&0&-1_l
\end{array}\right)\\
&=\frac{i}{2}
\left(\begin{array}{cccccc}
0&0&0&0&0&0\\
0&-1&0&0&0&0\\
0&0&0&0&0&0\\
0&0&0&0&0&0\\
0&0&0&0&1&0\\
0&0&0&0&0&0
\end{array}\right),
Z_{p,p}=\frac{i}{2}\left(\begin{array}{cccccc}
1_l&0&0&0&0&0\\
0&1&0&0&0&0\\
0&0&1_l&0&0&0\\
0&0&0&-1_l&0&0\\
0&0&0&0&-1&0\\
0&0&0&0&0&-1_l
\end{array}\right),
\end{align*}
$\langle \iota_2 d_p Z_{1,1}-d_{p,p}Z_{1,1}, Z_{p,p}\rangle=4p\mbox{Trace}( (\iota_2 d_p Z_{1,1}-d_{p,p}Z_{1,1}) Z_{p,p})=2p\neq 0$. 
Applying Theorem \ref{diagd} we see that $\iota_2$ is not tight for $p$ odd.
Besides these inclusions, identity homomorphisms and the representations corresponding to regular subalgebras we have two classes of irreducible (H2)-representations \cite{A6},\cite{A7}.\\
The first class we have are spin representations of $so(p,2)$ into $su(p^{'},p^{'})$. Here $p^{'}=2^{\frac{p}{2}-1}$ for $p$ even and $p^{'}=2^{\frac{p-1}{2}}$ for $p$ odd. For each even $p$ we have two spin representations, and for $p$ odd we have one.
Satake showed that the image of these representations is contained in a copy of $sp(2p^{'})$ if $p\equiv 1,2,3 (8)$ and in a copy of $so^*(2p^{'})$ if $ p \equiv 5,6,7 (8)$.
\begin{prop}
The spin representations 
\begin{align*}
&\rho\colon so(p,2)\rightarrow su(p^{'},p^{'}),\\
&\rho\colon so(p,2)\rightarrow sp(2p^{'}),\\
&\rho\colon so(p,2)\rightarrow so^*(2p^{'})
\end{align*}
are all tight.
\end{prop}
\begin{proof}
 Since all these representations are between tube type domains and satisfy (H2) they are tight by Theorem \ref{tube}.
\end{proof}
The second class of irreducible (H2)-representations are the skewsymmetric tensor representations of $su(p,1)$. Since the symmetric space corresponding to $su(p,1)$ is not of tube type we will have to resort to Theorem \ref{diagd} and do some calculations. We begin by describing these representations in detail.

As mentioned $su(p,1)$ acts on $(V,F)$  where $V$ is a vector space over $\mathbb{C}$ of dim $p+1$ and $F$ is a Hermitian form of signature $(p,1)$. We can extend this action to $\bigwedge^m V$, the exterior product of $m$ copies of $V$. We extend $F$ to a Hermitian form $F^m$ on $\bigwedge^m V$ defined as $F^m(x_1\wedge...\wedge x_m,y_1\wedge...\wedge y_m,):=det(F(x_i,y_j))$. 
If we choose an orthonormal basis $\{e_1,..,e_{p+1}\}$ for $(V,F)$,  $\{e_{i_1}\wedge ...\wedge e_{i_m}, i_1<...<i_m\}$ is an orthonormal basis for $(\bigwedge^m V,F^m)$. To shorten notation we write $e_I= e_{i_1}\wedge ...\wedge e_{i_m}$ for an ordered set $I=\{i_1,...,i_m\}$. We get 
\begin{equation}
F^m(e_I,e_I)=
\begin{cases}
1 \, , \mbox{ if } i_m\leq p\\
-1\, , \mbox{ if } i_m=p+1. 
\end{cases}
\end{equation}
Thus $F^m$ has signature $(\binom{p}{m},\binom{p}{m-1})=:(p^{'},q^{'})$. The extension of the action thus defines a representation 
\begin{equation}
\rho_m\colon su(p,1)\rightarrow su(p^{'},q^{'}).
\end{equation}

\begin{prop}
The skewsymmetric tensor representations $\rho_m \colon su(p,1)\rightarrow su(p^{'},q^{'})$ are not tight except for $m=1,p.$
\end{prop}

 \begin{proof}
To see if this is tight we have to calculate $\langle\rho_m d_{p,1}Z_{1,1}-d_{p^{'},q^{'}}Z_{1,1},Z_{p^{'},q^{'}}\rangle $ and apply Theorem \ref{diagd}. By \ref{disk} we have 
$$d_{p,1}Z_{1,1}=
\frac{i}{2}\left(\begin{array}{ccc}
0&0&0\\
0&1&0\\
0&0&-1
\end{array}\right)$$ where the zeroes denote zero matrices of appropriate size. Let us see how this matrix acts on the basis elements of $\bigwedge^m V$.\\
\begin{eqnarray*}
\rho_m d_{p,1} Z_{1,1}(e_I)&=&( d_{p,1} Z_{1,1}(e_{i_1})\wedge...\wedge e_{i_m})+...+(e_{i_1}\wedge...\wedge  d_{p,1} Z_{1,1}(e_{i_m}))\\
&=&\begin{cases}
\frac{i}{2}e_I \, , \mbox{ if } i_m=p\\
-\frac{i}{2}e_I\, , \mbox{ if } i_m=p+1 \mbox{ and } i_{m-1}<p\\
0 \mbox{ otherwise}
\end{cases}
\end{eqnarray*}\\
We decompose $\bigwedge^m V$ as $\bigwedge^m V=W^+\oplus W^-=W^+_0\oplus W^+_1\oplus W^-_0\oplus W^-_1$.
Here $W^\pm$ denotes the positive respectively the negative part of $\bigwedge^m V$ with respect to $F^m$ and the subscript $0$ and $1$ denotes the kernel of $\rho_m d_{p,1}Z_{1,1}$  and its complement. With respect to this decomposition we can write in matrix form\\
$\rho_m d_{p,1}Z_{1,1}=
\frac{i}{2}\left(\begin{array}{cccc}
0&0&0&0\\
0&1_{\binom{p-1}{m-1}}&0&0\\
0&0&0&0 \\
0&0&0&-1_{\binom{p-1}{m-1}}
\end{array}\right)$

and, assuming $p^{'}\geq q^{'}$, \\$d_{p^{'},q^{'}}Z_{1,1}=
\frac{i}{2}\left(\begin{array}{ccc}
0&0&0\\
0&1_{q^{'}}&0\\
0&0&-1_{q^{'}}
\end{array}\right).$

We thus get  $\langle \rho_m d_{p,1}Z_{1,1}-d_{p^{'},q^{'}}Z_{1,1},Z_{p^{'},q^{'}}\rangle =-(\binom{p+1}{m}+1)(\binom{p-1}{m-1}-\binom{p}{m-1})$. This is non-zero for all $m$  except $m=1,p$. The case $p^{'}<q^{'}$ has the same result via a slightly different calculation.

\end{proof}
If we choose $m=\frac{p+1}{2}$ for $p$ odd, we have $p^{'}=\binom{p}{\frac{p+1}{2}}=\binom{p}{\frac{p-1}{2}}=q^{'}$. Satake showed that in this case the image is contained in a copy of $sp(2p^{'})$ if $p\equiv 1 \, (4)$ and in a copy of  $so^*(2p^{'})$ if $p\equiv 3 \, (4)$ \cite{A7}.

\begin{prop}
For $p\equiv 1 \, (4)$ the skewsymmetric tensor representations  $\rho_m \colon su(p,1)\rightarrow sp(2p^{'})$ are not tight.
For $p\equiv 3 \, (4)$ the skewsymmetric tensor representations  $\rho_m \colon su(p,1)\rightarrow so^*(2p^{'})$ are not tight except for $p=3$.
\end{prop}
\begin{proof}
Since $sp(2p^{'})$ is tightly embedded in $su(p^{'},p^{'})$ the restriction of $\rho_m$ to $\rho_m\colon su(p,1)\rightarrow sp(2p^{'})$ can not be tight since a non-tight embedding can not be factored as two tight ones by Lemma \ref{lma2}. The same argument is valid for the representations into $so^*(2p^{'})$ when $p^{'}$ is even. 

When $p^{'}$ is odd however, $so^*(2p^{'})$ is not tightly embedded in $su(p^{'},p^{'})$. Thus we have to examine this case further.
We begin by showing how the image of $su(p,1)$ lies in $so^*(2p^{'})$. We do this by defining a bilinear form $B$ on $\bigwedge^m V$ by  $x\wedge y=B(x,y)e_1\wedge...\wedge e_{p+1}$. For $g\in GL(V)$ we have $B(\rho_m(g)x,\rho_m(g)y)=det(g)B(x,y)$ so clearly $B$ is invariant under $\rho_m(su(p,1))$. We now choose a basis for $\bigwedge^m V$ as follows. Let $k=\binom{p}{m}$, and choose some ordering $I_1,...,I_k$ of the $m$-subsets of $\{1,...,p\}$. Denote by $I_i^c$ the complement of $I_i$ in $\{1,...,p+1\}$. We then choose the ordered basis $e_{I_1},...,e_{I_k},\sigma_1 e_{I_1^c},...,\sigma_k e_{I_k^c}$ for $\bigwedge^m V$. The $\sigma_i$ chosen as $+1$ or $-1$ so that $B(e_{I_i},\sigma_i e_{I_i^c})$=1. With respect to this basis we can represent $F^m$ as the matrix 
$\left(\begin{array}{cc}
I&0\\
0&-I
\end{array}\right) $ and $B$ as the matrix 
$\left(\begin{array}{cc}
0&I\\
I&0
\end{array}\right) $.
These are the matrices in (\ref{soett}), (\ref{sotva}) defining our choice of matrix description of $so^*(2p^{'})$ . If we further assume we ordered the $I_i$:s so that the last $\binom{p-1}{m-1}$ indices correspond to subsets such that $i_m=p$, we have 
$\rho_m d_{p,1}Z_{1,1}=
\frac{i}{2}\left(\begin{array}{cccc}
0&0&0&0\\
0&1_{\binom{p-1}{m-1}}&0&0\\
0&0&0&0 \\
0&0&0&-1_{\binom{p-1}{m-1}}
\end{array}\right)$ and

$
d_{p^{'}}Z_{1,1}=
\frac{i}{2}\left(\begin{array}{cccccc}
1_l&0&0&0&0&0\\
0&0&0&0&0&0\\
0&0&1_l&0&0&0\\
0&0&0&-1_l&0&0\\
0&0&0&0&0&0\\
0&0&0&0&0&-1_l
\end{array}\right)
$ where $l=[\frac{p^{'}}{2}]$.\\
Now $\langle \rho_m d_{p,1}Z_{1,1}-d_{p^{'}}Z_{1,1},J\rangle =(2p^{'}-2)(\binom{p-1}{m-1}-p^{'}+1)\neq 0 $ except for $p=3$. Hence we have a tight representation of $su(3,1)$ into $so^*(6)$ but for no other values of $p$. This is however just one of the special isomorphisms between the simple Lie algebras in lower dimension.
\end{proof}

\section*{Acknowledgments}

I would like to thank my advisor Genkai Zhang for valuable discussions
in the preparation of this paper.

\end{document}